\documentclass[graybox]{svmult}

\usepackage{framed}
\usepackage{mathptmx}       
\usepackage{helvet}         
\usepackage{courier}        
\usepackage{type1cm}        
\usepackage{makeidx}         
\usepackage{graphicx}        
\usepackage{multicol}        
\usepackage[bottom]{footmisc}
\usepackage{subfigure}
\usepackage{amsmath, amssymb}
\usepackage[algo2e,linesnumbered,ruled]{algorithm2e}
\usepackage{cite}

\newtheorem{assumption}[theorem]{Assumption}

\newtheorem{thm}{Theorem}

\DeclareMathOperator*{\infm}{inf}
\DeclareMathOperator*{\minimize}{minimize}

\DeclareMathOperator*{\subjt}{subject \ to}
\DeclareMathOperator*{\sgn}{sign}

\newcommand{\bi}{\begin{itemize}}
\newcommand{\ei}{\end{itemize}}
\newcommand{\bd}{\begin{displaymath}}
\newcommand{\ed}{\end{displaymath}}
\newcommand{\be}{\begin{eqnarray*}}
\newcommand{\ee}{\end{eqnarray*}}

\makeindex             

\begin{document}

\title*{Data-Driven Nonlinear Stabilization Using Koopman Operator}
\titlerunning{Data-Driven Nonlinear Stabilization}
\author{Bowen Huang, Xu Ma, and Umesh Vaidya}
\authorrunning{B.~Huang, X.~Ma, and U.~Vaidya}
\institute{Bowen Huang \at Iowa state university, 3113 Coover Hall, Ames, IA, \email{bowen@iastate.edu}
\and Xu Ma \at Iowa state university, 3113 Coover Hall, Ames, IA, \email{maxu@alumni.iastate.edu}
\and Umesh Vaidya \at Iowa state university, 3215 Coover Hall, Ames, IA, \email{ugvaidya@iastate.edu}
}
%
%
\maketitle

\abstract{We propose the application of Koopman operator theory for the design of stabilizing feedback controller for a nonlinear control system. The proposed approach is data-driven and relies on the use of time-series data generated from the control dynamical system for the lifting of a nonlinear system in the Koopman eigenfunction coordinates. In particular, a finite-dimensional bilinear representation of a control-affine nonlinear dynamical system is constructed in the Koopman eigenfunction coordinates using time-series data. Sample complexity results are used to determine the data required to achieve the desired level of accuracy for the approximate bilinear representation of the nonlinear system in Koopman eigenfunction coordinates. 
A control Lyapunov function-based approach is proposed for the design of stabilizing feedback controller, and the principle of inverse optimality is used to comment on the optimality of the designed stabilizing feedback controller for the bilinear system.   A systematic convex optimization-based formulation is proposed for the search of control Lyapunov function. Several numerical examples are presented to demonstrate the application of the proposed data-driven stabilization approach. 
}


\section{Introduction}
Providing a systematic procedure for the design of stabilizing feedback control for a general nonlinear system will have a significant impact on a variety of application domains. The lack of proper structure for a general nonlinear system makes this design problem challenging. There have been several attempts to provide such a systematic approach, including convex optimization-based Sum-of-Squares (SoS) programming \cite{SOS_book, Parrilothesis} and differential geometric-based feedback linearization control \cite{sastry2013nonlinear,astolfi2015feedback}. The introduction of operator theoretic methods from the ergodic theory of dynamical systems provides another opportunity for the development of systematic methods for the design of feedback controllers \cite{Lasota}. The operator theoretic methods provide a linear representation for a nonlinear dynamical system. This linear representation of the nonlinear system is made possible by shifting the focus from state space to space of functions using two linear and dual operators, namely, the Perron-Frobenius (P-F) and Koopman operators. The work involving the third author \cite{VaidyaMehtaTAC, Vaidya_CLM,raghunathan2014optimal} provided a systematic linear programming-based approach involving transfer P-F operator for the optimal control of nonlinear systems. This contribution was made possible by exploiting the linearity and the positivity properties of the P-F operator. 

More recently, there has been increased research activity on the use of Koopman operator for the analysis and control of nonlinear systems \cite{Meic_model_reduction,mezic_koopmanism,susuki2011nonlinear,kaiser2017data,surana_observer,peitz2017koopman,mauroy2016global,surana2018koopman}. This recent work is mainly driven by the ability to approximate the spectrum (i.e., eigenvalues and eigenfunctions) of the Koopman operator from time-series data \cite{rowley2009spectral, DMD_schmitt, EDMD_williams, Umesh_NSDMD}. The data-driven approach for computing the spectrum of the Koopman operator is attractive as it opens up the possibility of employing operator theoretic methods for data-driven control. Research works in \cite{kaiser2017data,peitz2017koopman,korda2018linear,korda2018power,arbabi2018data,hanke2018koopman,sootla2018optimal} are proposing to develop Koopman operator-based data-driven methods for the design of optimal control and model predictive control for nonlinear and partial differential equations as well. The existing approaches rely on identification of linear predictors and the use of linear control design techniques for Koopman-based control. However, the tightness of these linear predictors cannot be theoretically guaranteed. In comparison, this book chapter proposes data-driven identification and bilinear representation of nonlinear control systems in Koopman eigenfunction coordinates. The bilinear representation is tight and theoretically justified in the sense that in the limit as the number of basis function approaches infinity, the finite-dimensional bilinear representation will approach the true lifting of a control system in the function space. To address the control design problem of a more complex bilinear system, we propose a control Lyapunov function-based approach for feedback stabilization. Furthermore, sample complexity results from \cite{sample_complexity} are used to characterize the relationship between the amount of training data and the approximation error of our bilinear predictor. The work in this book chapter is the extended version of the work presented in \cite{bowen_koopmanstabilziationCDC}, where the data-driven identification for control component is new.

The main contributions of the book chapter are as follows. We present a data-driven approach for feedback stabilization of a nonlinear system (refer to Fig. \ref{fig_schematic}). We first show that the nonlinear control system can be identified from the time-series data generated by the system for two different input signals, namely zero input and step input. For this identification, we make use of linear operator theoretic framework involving Fokker Planck equation. Furthermore, sample complexity results developed in \cite{sample_complexity} are used to determine the data required to achieve the desired level for the approximation. 
This process of identification leads to a finite-dimensional bilinear representation of the nonlinear control system in Koopman eigenfunction coordinates. This finite-dimensional approximation of the bilinear system is used for the design of a stabilizing feedback controller. While the control design for a bilinear system is, in general, a challenging problem, we propose a systematic approach based on the theory of control Lyapunov function (CLF) and inverse optimality for feedback control design \cite{Khalil_book}. While the search for CLFs for a general nonlinear system is a difficult problem, we use a bilinear representation of the nonlinear control system in the Koopman eigenfunction space to search for a CLF for the bilinear system. By restricting the search of CLFs to a class of quadratic Lyapunov functions, we can provide a convex programming-based systematic approach for determining the CLF \cite{Boyd_book}. The principle of inverse optimality allows us to connect the CLF to an optimal cost function. The controller designed using CLF also optimizes an appropriate cost. Using this principle, we comment on the optimality of the controller designed using CLF.

The main contributions of this work are as follows. We present a data-driven approach for the identification and representation of a nonlinear control system as a bilinear system. The bilinear structure of the control dynamical system is exploited to provide a systematic approach for the feedback stabilization of nonlinear systems. The proposed systematic approach relies on control Lyapunov function (CLF) and quadratic stabilization in Koopman eigenfunction space. A convex optimization-based formulation is proposed for searching quadratic CLFs. The CLF is used to propose a different formula for the stabilizing feedback control. One of them is the Sontag formula which allows us to comment on the optimality of the designed stabilizing feedback controller using the principle of inverse optimality. 

\begin{figure}[htbp]
\centering
\includegraphics[width=1.0\linewidth]{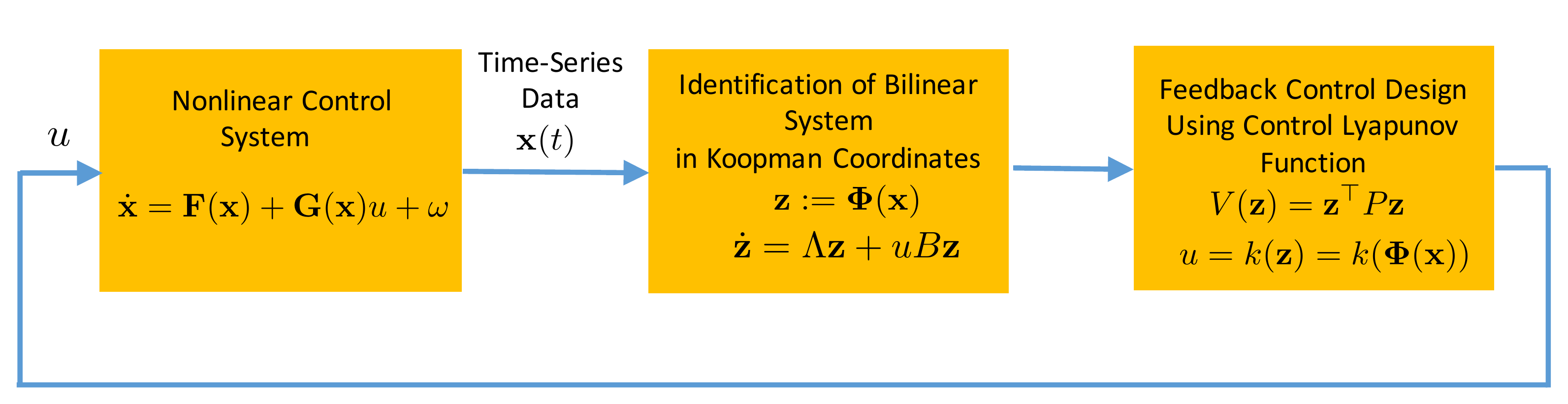}
\caption{Data-Driven Identification and Control of Nonlinear System}
 \label{fig_schematic}
\end{figure}
This book chapter is organized as follows. In Section \ref{section_prelim}, we present some preliminaries on the Koopman operator, Fokker Planck equation, and control Lyapunov functions.  In Section \ref{section_bilinear}, we present the identification scheme for the data-driven identification of a nonlinear control system as a bilinear system in Koopman eigenfunction coordinates. In Section \ref{section_main}, a convex optimization-based formulation is proposed to search for quadratic CLFs and for the design of stabilizing feedback controller.  Simulation results are presented in Section  \ref{section_simulation}, followed by conclusion in Section \ref{section_conclusion}.

\section{Preliminaries}\label{section_prelim}

In this section, we present some preliminaries on the Koopman operator, Fokker Planck equation, and control Lyapunov function-based approach on the design of stabilizing feedback controllers for nonlinear systems. 
\subsection{Koopman Operator}
Consider a continuous-time dynamical system of the form 
\begin{eqnarray}
\dot{\vec x}=\vec F(\vec x)\label{system}
\end{eqnarray}
where $\vec x\in X\subset \mathbb{R}^n$ and the vector field $\vec F$ is assumed to be continuously differentiable. Let $\vec S(t,\vec x_0)$ be the solution of the system (\ref{system}) starting from initial condition $\vec x_0$ and at time $t$. Let $\cal O$ be the space of all observables $f: X\to \mathbb{C}$. 
\begin{definition}[Koopman operator] The Koopman semigroup of operators $U_t :{\cal O}\to {\cal O}$ associated with system (\ref{system}) is defined by
\end{definition}
\begin{eqnarray}
[U_t f](\vec x)=f(\vec S(t,\vec x)).
\end{eqnarray}
It is easy to observe that the Koopman operator is linear on the space of observables although the underlying dynamical system is nonlinear.  In particular, we have
\[[U_t (\alpha f_1 +f_2)](\vec x)=\alpha [U_t f_1](\vec x)+[U_t f_2](\vec x).\]

Under the assumption that the function $f$ is continuously differentiable, the semigroup $[U_t f](\vec x)=\rho(\vec x,t)$ can be obtained as the solution of the following partial differential equation
\[\frac{\partial \rho}{\partial t}=\vec F\cdot \nabla \rho=: L \rho\]
with initial condition $\rho(\vec x,0)=f(\vec x)$. From the semigroup theory it is known \cite{Lasota} that the operator $L$ is the infinitesimal generator for the Koopman operator, i.e., 
\[L \rho=\lim_{t\to 0}\frac{U_t \rho-\rho}{t}.\]

The linear nature of Koopman operator allows us to define the eigenfunctions and eigenvalues of this operator as follows.
\begin{definition}[Koopman eigenfunctions] The eigenfunction of Koopman operator is a function $\phi_\lambda\in {\cal O}$ that satisfies
\begin{eqnarray}
[U_t \phi_\lambda](x)=e^{\lambda t} \phi_\lambda(x)
\end{eqnarray}
for some $\lambda\in \mathbb{C}$. The $\lambda$ is the associated eigenvalue of the Koopman eigenfunction and is assumed to belong to the point spectrum. 
\end{definition}

The spectrum of the Koopman operator is far more complex than simple point spectrum and could include continuous spectrum \cite{Meic_model_reduction}. The eigenfunctions can also be expressed in terms of the infinitesimal generator of the Koopman operator $L$ as follows
\[L \phi_\lambda =\lambda \phi_\lambda.\]
The eigenfunctions of Koopman operator corresponding to the point spectrum are smooth functions and can be used as coordinates for linear representation of nonlinear systems.  
\subsection{Fokker Planck Equation}
We need the preliminaries on Fokker Planck equation for the purpose of data-driven identification of nonlinear control system. Consider a nonlinear dynamical system perturbed with white noise process. 
\begin{eqnarray}
\dot {\vec x}={\vec F}({\vec x})+{\vec \omega}\label{sde}
\end{eqnarray}
where ${\vec \omega}$ is the white noise process. Following assumption is made on the vector function $\vec F$.
\begin{assumption}\label{assume_diff}
Let $\vec F =(\vec F_1,\ldots, \vec F_n)^\top$. We assume that the functions $\vec F_i$ $i=1,\ldots n$ are ${\vec C}^4$ functions. 
\end{assumption}
We assume that the distribution of ${\vec x}(0)$ is absolutely continuous and has density $p_0({\vec x})$.   Then we know that ${\vec x}(t)$ has a density $p({\vec x},t)$ which satisfies following Fokker-Planck (F-P) equation also known as Kolomogorov forward equation.
\begin{eqnarray}
\frac{\partial p({\vec x},t)}{\partial t}=-  \nabla \cdot\left( {\vec F}({\vec x})p({\vec x},t)\right) +\frac{1}{2}\nabla^2 p({\vec x},t)
\end{eqnarray}
Following Assumption \ref{assume_diff}, we know the solution $p({\vec x},t)$ to F-P equation exists and is differentiable (Theorem 11.6.1 \cite{Lasota}). Under some regularity assumptions on the coefficients of the F-P equation (Definition 11.7.6 \cite{Lasota}) it can be shown that the F-P admits a generalized solution. The generalized solution is used in defining stochastic semi-group of operators ${\vec \{\mathbb{P}}_t\}_{t\geq 0}$ such that 
\begin{eqnarray}
[\mathbb{P}_t p_0]({\vec x})=p({\vec x},t).
\end{eqnarray}
Furthermore, the right hand side of the F-P equation is the infinitesimal generator for stochastic semi-group of operators $\mathbb{P}_t$ i.e.,

\begin{eqnarray}
\mathbb{A} \varphi=\lim_{t\to 0}\frac{(\mathbb{P}_t^s-I)\varphi}{t}
\end{eqnarray}
where 
\[\mathbb{A}\varphi :=-  {\vec \nabla} \cdot\left( ({\vec F}({\vec x}) \varphi)\right) +\frac{1}{2}\nabla^2 \varphi\]
Let $\psi(x)\in {\vec C}^2(\mathbb{R}^n)$ be an observable we have 
\begin{eqnarray}
\frac{d}{dt}\int p({\vec x},t) \psi({\vec x}) dx=\int {\mathbb{A}}p({\vec x},t) \psi({\vec x})d{\vec x}\nonumber\\=\int p({\vec x},t)\mathbb{A}^*\psi({\vec x})d{\vec x}
\end{eqnarray}
where $\mathbb{A}^*$ is adjoint to $\mathbb{A}$ and is defined as \begin{equation}\label{eq:Koopmangenerator}
\mathbb{A}^*\psi={\vec F}\cdot \nabla \psi+\frac{1}{2}\nabla^2 \psi
\end{equation}
The semi-group corresponding to the $\mathbb{A}^*$ operator is given by 
\begin{eqnarray}\mathbb{A}^*\psi =\lim_{t\to 0}\frac{(\mathbb{U}_t-I)\psi}{t}\label{generatorstoc_koopman}
\end{eqnarray}
where 
\begin{equation}\label{eq:Koopmangroup}
[\mathbb{U}_t\psi]({\vec x})=\mathbb{E}[\psi({\vec x}(t))\mid {\vec x}(0)={\vec x}].
\end{equation}

For the deterministic dynamical system $\dot {\vec x}={\vec F}({\vec x})$, i.e., in the absence of noise term, the above definitions of generators and semi-groups reduces to Perron-Frobenius and Koopman operators. In particular, the propagation of probability density function capturing uncertainty in initial condition is given by Perron-Frobenius (P-F) operator and is defined as follows.
\begin{definition}
The P-F operator for deterministic dynamical system $\dot {\vec x}={\vec F}({\vec x})$ is defined as follows
\begin{eqnarray}
[P_t p_0]({\vec x})=p_0(\vec S(-t,\vec x))\left|\frac{\partial \vec S(-t,\vec x)}{\partial {\vec x}}\right|
\end{eqnarray}
where $\vec S(t,\vec x)$ be the solution of the system (\ref{system}) starting from initial condition $\vec x$ and at time $t$, and $\left|\cdot \right|$ stands for the determinant. 
\end{definition}
The infinitesimal generator for the P-F operator is given by
\begin{eqnarray}
A \varphi := -\nabla \cdot ({\vec F}({\vec x}) \varphi)=\lim_{t\to 0}\frac{(P_t-I)\varphi}{t}
\end{eqnarray}

\subsection{Feedback Stabilization and Control Lyapunov Functions}
For the simplicity of the presentation, we will consider only the case of single input in this paper. All the results carry over to the multi-input case in a straightforward manner. Consider a single input control affine system of the form.
\begin{align}
\label{non_lin_sys}
\dot{\vec x} = \vec F(\vec x)+\vec G(\vec x)u, 
\end{align}
where $\vec x(t) \in \mathbb{R}^n$ denotes the state of the system, $u(t) \in \mathbb{R}$ denotes the single input of the system, and $\vec F, \vec G: \mathbb{R}^n \rightarrow \mathbb{R}^n$ are assumed to be continuously differentiable mappings. We assume that $\vec F(\vec0) = \vec0$ and the origin is an unstable equilibrium point of the uncontrolled system $\dot{\vec x}=\vec F(\vec x)$.

The \textit{state feedback stabilization} problem associated with system (\ref{non_lin_sys}) seeks a possible feedback control law of the form
\begin{align*}
u = k(\vec x)
\end{align*}
with $k: \mathbb{R}^n \rightarrow \mathbb{R}$ such that $\vec x = \vec 0$ is asymptotically stable within some domain $\mathcal{D} \subset \mathbb{R}^n$ for the closed-loop system
\begin{align}
\dot{\vec x} = \vec F(\vec x)+\vec G(\vec x)k(\vec x). \label{closed_loop}
\end{align}

One of the possible approaches for the design of stabilizing feedback controllers for the nonlinear system (\ref{non_lin_sys}) is via control Lyapunov functions that are defined as follows.

\begin{definition}
Let $\mathcal{D} \subset \mathbb{R}^n$ be a neighborhood that contains the equilibrium $\vec x = \vec0$. A \textit{control Lyapunov function} (CLF) is a continuously differentiable positive definite function $V: \mathcal{D} \rightarrow \mathbb{R}_+$ such that for all $\vec x \in \mathcal{D} \setminus \{\vec0\}$ we have
\end{definition}
\begin{align*}
\infm_u \ \left[ \frac{\partial V}{\partial x} \cdot \vec F(\vec x) + \frac{\partial V}{\partial x} \cdot \vec G(\vec x)u \right] := \infm_u \ \Big[ V_x\vec F(\vec x) + V_x\vec G(\vec x) u \Big] <0
\end{align*}

It has been shown in \cite{artstein1983stabilization, sontag1989universal} that the existence of a CLF for system (\ref{non_lin_sys}) is equivalent to the existence of a stabilizing control law $u = k(\vec x)$ which is almost smooth everywhere except possibly at the origin $\vec x = \vec 0$.

\begin{theorem} [see \cite{astolfi2015feedback}, Theorem 2] \label{clf_thm}
There exists an almost smooth feedback $u = k(\vec x)$, i.e., $k$ is continuously differentiable for all $\vec x \in \mathbb{R}^n \setminus \{\vec0\}$ and continuous at $\vec x = \vec0$, which globally asymptotically stabilizes the equilibrium $\vec x = \vec0$ for system (\ref{non_lin_sys}) if and only if there exists a radially unbounded CLF $V(\vec x)$ such that
\begin{enumerate}
\item For all $\vec x \neq \vec0$, $V_x\vec G(\vec x) = 0$ implies $V_x\vec F(\vec x) < 0$;
\item For each $\varepsilon > 0$, there is a $\delta > 0$ such that $\|\vec x\| < \delta$ implies the existence of a $|u| < \varepsilon$ satisfying $V_x\vec F(\vec x) + V_x\vec G(\vec x) u < 0$.
\end{enumerate}
\end{theorem}

In the theorem above, condition 2) is known as the small control property, and it is necessary to guarantee continuity of the feedback at $x \neq 0$. If both conditions 1) and 2) hold, an almost smooth feedback can be given by the so-called Sontag's formula
\begin{align}
\label{sontag} k(\vec x) := \begin{cases} -\frac{V_x\vec F + \sqrt{(V_x\vec F)^2 + (V_x\vec G)^4}}{V_x\vec G} & \text{if } V_x\vec G(\vec x) \neq 0 \\ 0 & \text{otherwise.} \end{cases}
\end{align}

Besides Sontag's formula, we also have several other possible choices to design a stabilizing feedback control law based on the CLF given in Theorem \ref{clf_thm}. For instance, if we are not constrained to any specifications on the continuity or amplitude of the feedback, we may simply choose
\begin{align}\label{control2}
k(\vec x) &:= -K \sgn\big[ V_x\vec G(\vec x) \big] \\
\label{control1}
k(\vec x) &:= -K V_x\vec G(\vec x)
\end{align}
with some constant gain $K > 0$. Then, differentiating the CLF with respect to time along trajectories of the closed-loop (\ref{closed_loop}) yields
\begin{align*}
\dot V & = V_x\vec F(\vec x) - K \big| V_x\vec G(\vec x) \big| \\
\dot V & = V_x\vec F(\vec x) - K V_x\vec G(\vec x)^2.
\end{align*}
Hence, by the stabilizability property of condition 1), there must exist some $K$ large enough such that $\dot V < 0$ for all $\vec x \neq \vec0$, because whenever $V_x\vec F(\vec x) \geq 0$ we have $V_x\vec G(\vec x) \neq 0$.

On the other hand, the CLFs also enjoy some optimality property using the principle of inverse optimal control. In particular, consider the following optimal control problem
\begin{align}
\minimize_u \quad & \int_0^\infty (q(x) + u^\top u)dt \label{cost} \\
\subjt \quad & \dot{\vec x} = \vec F(\vec x)+\vec g(\vec x)u \nonumber
\end{align}
for some continuous, positive semidefinite function $q: \mathbb{R}^n \rightarrow \mathbb{R}$. Then the modified Sontag's formula
\begin{align}
\label{mod_sontag} k(\vec x) := \begin{cases} -\frac{V_x\vec F + \sqrt{(V_x\vec F)^2 + q(x)(V_x\vec G)^2}}{V_x\vec G} & \text{if } V_x\vec G(\vec x) \neq 0 \\ 0 & \text{otherwise} \end{cases}
\end{align}
builds a strong connection with the optimal control. In particular, if the CLF has level curves that agree in shape with those of the value function associated with cost (\ref{cost}), then the modified Sontag's formula (\ref{mod_sontag}) will reduce to the optimal controller \cite{freeman1996control, primbs1999nonlinear}.

\section{Data-driven Identification of Nonlinear System  }\label{section_bilinear}

In this section we discuss the application of linear operator theoretic framework  for the identification of nonlinear dynamical system in the Koopman eignfunctions space. Consider the control dynamical system perturbed by stochastic noise process. 
\begin{eqnarray}\dot {\vec x}={\vec F}({\vec x})+{\vec G}({\vec x})u+{\vec\omega}\label{control_sde}
\end{eqnarray}
where ${\vec \omega}\in \mathbb{R}^n$ is the white noise process.  The presence of noise term is essential to ensure persistency of  excitation for the purpose of identification. Following assumption is made on the vector functions $\vec F$ and $\vec G$. 
\begin{assumption}\label{assume1}
Let $\vec F =(\vec F_1,\ldots, \vec F_n)^\top$ and $\vec G=(\vec G_1,\ldots, \vec G_n)^\top$.  We assume that the functions $\vec F_i$ and $\vec G_i$ for $i=1,\ldots n$ are ${\vec C}^4$ functions. 
\end{assumption}

The objective is to identify the nonlinear vector fields $\vec F$ and $\vec G$ using the time-series data generated by the control dynamical system and arrive at a continuous-time dynamical system of the form
\begin{eqnarray}
\dot {\vec z}=\Lambda {\vec z}+u B{\vec  z}\label{sys_bilinear_cont}
\end{eqnarray}
where ${\vec z}\in \mathbb{R}^N$ with $N\geq n$. We now make following assumption on the control dynamical system (\ref{control_sde}).
\begin{assumption}
We assume that all the trajectories of the control dynamical system (\ref{control_sde})  starting from different initial conditions for control input $u=0$ and for step input remains bounded. 
\end{assumption}
\begin{remark}
This assumption is essential to ensure that the control dynamical system can be identified from the time-series data generated by the system for two different inputs signals. 
\end{remark}
The goal is to arrive at a continuous-time bilinear representation of the nonlinear control system (\ref{control_sde}). Towards this goal we assume that the time-series data from the continuous time dynamical system (\ref{control_sde}) is available for two different control input namely zero input and step input. The discrete time-series data is generated from the continuous time dynamical system with sufficiently small discretization time step $\Delta t$ and this time-series data is represented as
\begin{eqnarray}
({\vec x}^s_{k+1},{\vec x}^s_k)\label{data}
\end{eqnarray}
The subscript $s$ signifies that the data is generated by dynamical system of the form 
\begin{eqnarray}\dot {\vec x}={\vec F}({\vec x})+{\vec G}({\vec x})s+{\vec\omega}\label{eq_s}
\end{eqnarray}
So that $s=0$ and $s=1$ corresponds to the case of zero input and step input respectively. 
Let
\[\Psi=[\psi_1,\ldots, \psi_N]\]
be the set of observables with $\psi_i:\mathbb{R}^n\to \mathbb{R}$. The time evolution of these observables under the continuous time control dynamical system with no noise can be written as 
\begin{eqnarray}\frac{d \Psi}{dt}&=&{\vec F}({\vec x})\cdot \nabla \Psi +u {\vec G}({\vec x})\cdot \nabla \Psi\nonumber\\
&=&{\cal A}\Psi +u{\cal B}\Psi
\end{eqnarray}
where ${\cal A}$ and $\cal B$ are linear operators. The objective is to construct the finite dimensional approximation of these linear operators, $\cal A$, and $\cal B$ respectively from time-series data to arrive at a finite dimensional approximation of control dynamical system as in Eq. (\ref{sys_bilinear_cont}). 

With reference to Eq. (\ref{eq:Koopmangenerator}), let $\mathbb{A}_1^*$ and $\mathbb{A}_0^*$ be the generator corresponding to the control dynamical system with step input i.e., $s=1$ and $s=0$ respectively in Eq. (\ref{eq_s}). We have 
\begin{eqnarray}(\mathbb{A}_1^*-\mathbb{A}_0^*)\psi={\vec G}({\vec x})\cdot \nabla \psi\label{g_approx}
\end{eqnarray}
Under the assumption that the sampling time $\Delta t$ between the two consecutive time-series data point is sufficiently small, the generators $\mathbb{A}_s^*$ can be approximated as 
\begin{eqnarray}\mathbb{A}_s^*\approx \frac{\mathbb{U}_{\Delta t}^s-I}{\Delta t}\label{f_approx}
\end{eqnarray}
Substituting for $s=1$ and $s=0$ in (\ref{f_approx}) and using (\ref{g_approx}), we obtain
\begin{eqnarray}
\frac{\mathbb{U}_{\Delta t}^1-\mathbb{U}_{\Delta t}^0}{\Delta t}\approx {\vec G}({\vec x})\cdot \nabla={\cal B}\label{B_approx}
\end{eqnarray}
and 
\begin{eqnarray}
\frac{\mathbb{U}_{\Delta t}^0-I}{\Delta t}\approx {\vec  F}({\vec x})\cdot \nabla={\cal A}\label{A_approx}
\end{eqnarray}
Using the time-series data generated from dynamical system  (\ref{eq_s}) for $s=0$ and $s=1$,  it is possible to construct the finite dimensional approximation of the operators $\mathbb{U}_{\Delta t}^0$ and $\mathbb{U}_{\Delta t}^1$ respectively thereby approximating the operators $\cal A$ and $\cal B$ respectively. In the following we explain the extended dynamic mode decomposition-based procedure for the approximation of these operators from time-series data.

\subsection{Finite Dimensional Approximation}
We use Extended Dynamic Mode Decomposition (EDMD) algorithm for the approximation of $\mathbb{U}_{\Delta t}^1$ and $\mathbb{U}_{\Delta t}^0$ thereby approximating $\cal A$ and $\cal B$ in Eqs. (\ref{A_approx}) and (\ref{B_approx}) respectively \cite{EDMD_williams}. For this purpose let the time-series data generated by the dynamical system (\ref{eq_s}) be given by
\begin{eqnarray}
\overline{\vec X} = [\vec x^s_1,\vec x^s_2,\ldots,\vec x^s_M],&\overline{\vec Y} = [\vec y^s_1,\vec y^s_2,\ldots,\vec y^s_M] \label{data}
\end{eqnarray}
where $\vec y^s_k=\vec x^s_{k+1}$ with $s=0$ or $s=1$ i.e., zero input and step input. Furthermore, let $\mathcal{H}=
\{\psi_1,\psi_2,\ldots,\psi_N\}$ be the set of dictionary functions or observables and ${\cal G}_{\cal H}$ be the span of $\cal H$.
The choice of dictionary functions is very crucial and it should be rich enough to approximate the leading eigenfunctions of the Koopman operator. Define vector-valued function $\varPsi:X\to \mathbb{C}^{N}$
\begin{equation}
\varPsi(\vec x):=\begin{bmatrix}\psi_1(\vec{x}) & \psi_2(\vec{x}) & \cdots & \psi_N(\vec{x})\end{bmatrix}^\top.
\end{equation}
In this application, $\varPsi$ is the mapping from state space to function space. Any two functions $f$ and $\hat{f}\in \mathcal{G}_{\cal H}$ can be written as
\begin{eqnarray}
f = \sum_{k=1}^N a_kh_k=\varPsi^\top \vec{a},\quad \hat{f} = \sum_{k=1}^N \hat{a}_kh_k=\varPsi^\top \vec{\hat{a}}
\end{eqnarray}
for some coefficients $\vec{a}$ and $\vec{\hat{a}}\in \mathbb{C}^N$. Let \[ \hat{f}(\vec{x})=[U^s_{\Delta t}f](\vec{x})+r\]
where $r$ is a residual function that appears because $\mathcal{G}_{\cal H}$ is not necessarily invariant to the action of the Koopman operator. To find the optimal mapping which can minimize this residual, let $\vec U$ be the finite dimensional approximation of the Koopman operator $U^s_{\Delta t}$. Then the  matrix $\vec U^s$ is obtained as a solution of least-squares problem as follows 
\begin{equation}\label{edmd_op}
\minimize_{{\vec U}^s} \quad \|{\vec G}^s{\vec U}^s-{\vec A}^s\|_F
\end{equation}
where
\begin{eqnarray}\label{edmd1}
{\vec G}^s=\frac{1}{M}\sum_{m=1}^M \varPsi(\vec{x}^s_m)^\top \varPsi(\vec{x}^s_m),\;\;\;
{\vec A}^s=\frac{1}{M}\sum_{m=1}^M \varPsi(\vec{x}^s_m)^\top \varPsi(\vec{y}^s_m)
\end{eqnarray}
with ${\vec U}^s,{\vec G}^s,{\vec A}^s\in\mathbb{C}^{N\times N}$. The optimization problem (\ref{edmd_op}) can be solved explicitly with a solution in the following form
\begin{eqnarray}
{\vec U}^s=({\vec G}^s)^\dagger {\vec A}^s\label{EDMD_formula}
\end{eqnarray}
where $({\vec G}^s)^{\dagger}$ denotes the psedoinverse of matrix ${\vec G}^s$.

Under the assumption that the leading Koopman eigenfunctions are contained within $\mathcal{G}_{\mathcal{H}}$, the eigenvalues of $\vec U$ are approximations of the Koopman eigenvalues. The right eigenvectors of ${\vec U}^{s=0}$ can be used then to generate the approximation of Koopman eigenfunctions. In particular, the approximation of Koopman eigenfunction is given by
\begin{equation}\label{EDMD_eigfunc_formula}
\phi_j=\varPsi ^\top v_j, \quad j = 1,\ldots,N
\end{equation}
where $v_j$ is the $j$-th right eigenvector of ${\vec U}^0$, and $\phi_j$ is the approximation of the eigenfunction of Koopman operator corresponding to the $j$-th eigenvalue, $\lambda_j\in \mathbb{C}$. 

The bilinear representation of nonlinear control dynamical system can be constructed either in the space of basis function ${\varPsi}$ or the eigenfunctions of the Koopman operator ${\varPhi}$, where 
\[\varPhi(\vec x):=[\phi_1(\vec x),\ldots, \phi_N(\vec x)]^\top.\]
In this work, we constructed the bilinear representation in the Koopman eigenfunctions coordinates. Towards this goal, we define

\[\hat{\varPhi}(\vec x):=[\hat \phi_1(\vec x),\ldots, \hat \phi_N(\vec x)]^\top\] where $\hat \phi_i:=\phi_i$ if $\phi_i$ is a real-valued eigenfunction and $\hat \phi_i:=2 {\rm Re}(\phi)$, $\hat \phi_{i+1}:=-2{\rm Im}(\phi_i)$, if $i$ and $i+1$ are complex conjugate eigenfunction pairs. Consider now the transformation as $\hat {\varPhi}: \mathbb{R}^n\to \mathbb{R}^N$ as 
\[\vec z=\hat{\varPhi}(\vec x).\]
Then in this new coordinates system Eq. (\ref{non_lin_sys}) takes the following form
\begin{eqnarray}
\dot{\vec z}=\Lambda \vec z+uB{\vec z} .\label{bilinear1}
\end{eqnarray}
where the matrix $\Lambda$ has a block diagonal form where the block corresponding to the eigenvalue $\hat\lambda_i$, such that $\Lambda_{(i,i)} =\hat\lambda_i$ if $\phi_i$ is real, and 
\begin{align}\label{eig_conversion}
\begin{bmatrix}\Lambda_{(i,i)}&\Lambda_{(i,i+1)}\\ \Lambda_{(i+1,i)}&\Lambda_{(i+1,i+1)}\end{bmatrix} =\lvert\lambda_i\rvert\begin{bmatrix}
\cos(\angle{\hat\lambda_i})&\sin(\angle{\hat\lambda_i})\\-\sin(\angle{\hat\lambda_i})&\cos(\angle{\hat\lambda_i})
\end{bmatrix}
\end{align}
if $\phi_i$ and $\phi_{i+1}$ are complex conjugate pairs.
The $\hat\lambda_i$ associated with the continuous time system dynamics. The relationship between discrete-time Koopman eigenvalues $\lambda_i$ and continuous time $\hat\lambda_i$ can be written as $\hat\lambda_i = \log(\lambda_i)/\Delta t$.

Similarly data generated using step for the control dynamical system is used to generate time-series data $\{{\vec x}_k^1\}$ and for the approximation of ${\vec U}^1$. The approximation of the operator ${\cal B}$ in the coordinates of basis functions, $\varPsi({\vec x})$ denoted by $\bar B$, and the eigenfunction coordinates $\hat {\varPhi}({\vec x})$ denoted by $B$ can be  obtained as follows:
\begin{eqnarray}
\bar B = \frac{\vec U^1-\vec U^0}{\Delta t},\;\;\;\;\;
B = V^\top{\bar B}(V^\top)^{-1}
\end{eqnarray}
where each column of $V$, $v_j$ is the $j$th eigenvector of $\vec U^0$.

There are two sources of error in the approximation of Koopman operator and its spectrum and will be reflected in the bilinear representation of nonlinear system namely in the $\Lambda$ and $B$ matrices. The first source of error is due to a finite number of basis functions used in the approximation of the Koopman operator. Under the assumption that the choice of basis functions is sufficiently rich and $N$ is large this approximation error is expected to be small. However, selection of basis function is a actively research topic with no agreement on the best choice of basis function for general nonlinear system. The second source of error, which is more relevant to this work, arise due to the finite length of data used in the approximation of the Koopman operator. Sample complexity results for control dynamical systems are developed in \cite{sample_complexity} to derive an analytical formula for the approximation of Koopman operator as the function of data length. We proved that the approximation error between the true Koopman operator and its approximation decreases as $\frac{1}{\sqrt{T}}$, where $T$ is the time length of the data. These sample complexity results are used to determine the data required to achieve the desired level of accuracy of the approximation. In particular, the bilinear representation of control dynamical system with approximation error due to the finite length of data explicitly accounted for can be written as 
\begin{eqnarray}
\dot {\vec z}=(\Lambda +\Delta \Lambda) {\vec z}+u(B+\Delta B){\vec z},\label{bilinear_uncertain}
\end{eqnarray}
where $\Delta \Lambda$ and $\Delta B$ are approximation error. Using sample complexity results discovered in \cite{sample_complexity}, we can determine the data length $M$ so that $\parallel \Delta \Lambda\parallel\leq \epsilon_\Lambda$ and $\parallel \Delta B\parallel\leq \epsilon_B$, with $\epsilon_{\Lambda}$ and $\epsilon_B$ being the predetermined acceptable bounds.



\section{Feedback Controller Design}\label{section_main}

The control Lyapunov function provides a powerful tool for the design of a stabilizing feedback controller which also enjoys some optimality property using the principle of inverse optimality. However, one of the main challenges is providing a systematic procedure to find CLFs. For a general nonlinear system finding a CLF remains a challenging problem. We exploit the bilinear structure of the nonlinear system in the Koopman eigenfunction space to provide a systematic procedure for computing control Lyapunov function. We restrict the search for the control Lyapunov function to the class of quadratic Lyapunov function of the form $V({\vec z})={\vec z}^\top P{\vec z}$. It is important to emphasize that although the Lyapunov function is restricted to be quadratic in Koopman eigenfunctions space $\vec z$, the Lyapunov function contains higher order nonlinearities in the original state space $\vec x$. Theorem \ref{bilin_stb} can be stated for the quadratic stabilization of the following bilinear control system.
\begin{eqnarray}
\dot {\vec z}=\Lambda  {\vec z}+u B{\vec z}\label{bilinear_sys}
\end{eqnarray}
In the sequel, if there exists a quadratic CLF for the bilinear system (\ref{bilinear_sys}), then we will say that the system (\ref{bilinear_sys}) is {\it quadratic stabilizable}.



\begin{thm}
\label{bilin_stb}
System (\ref{bilinear_sys}) is quadratic stabilizable if and only if there exists an $N \times N$ symmetric positive definite $P$ such that for all non-zero $\vec z \in \mathbb{R}^N$ with $\vec z^\top(P\Lambda+\Lambda^\top P)\vec z \geq 0$, we have $\vec z^\top(PB+B^\top P)\vec z \neq 0$.
\end{thm}

\begin{proof}
Sufficiency $(\Leftarrow)$: Suppose there is a symmetric, positive definite $P$ that satisfies the condition of Theorem \ref{bilin_stb}. We can use it to construct $V(\vec z) = \vec z^\top P\vec z$ as our Lyapunov candidate function, and the derivative of $V$ with respect to time along trajectories of (\ref{bilinear_sys}) is given by
\begin{align*}
\dot V & = \vec z^\top P\dot{\vec z} + \dot{\vec z}^\top P\vec z \nonumber \\
& = \vec z^\top(P\Lambda+\Lambda^\top P)\vec z + u\vec z^\top(PB+B^\top P)\vec z.
\end{align*}
Since for all $\vec z \neq 0$ we have $\vec z^\top(PB+B^\top P)\vec z \neq 0$ when $\vec z^\top(P\Lambda+\Lambda^\top P)\vec z \geq 0$, we can always find a control input $u(\vec z)$ such that
\begin{align*}
\dot V < 0, \quad \forall \vec z \in \mathbb{R}^N \setminus \{0\}.
\end{align*}
Therefore, $V(\vec z)$ is indeed a CLF for system (\ref{bilinear_sys}).

Necessity ($\Rightarrow$): We will prove this by contradiction. Suppose that system (\ref{bilinear_sys}) has a CLF in the form of $V(\vec z) = \vec z^\top P\vec z$, where $P$ does not satisfy the condition of Theorem \ref{bilin_stb}. That is, there exists some $\bar{\vec z} \neq 0$ such that ${\bar{\vec z}}^\top(P\Lambda+\Lambda^\top P)\bar{\vec z} \geq 0$ but $\bar{\vec z}^\top(PB+B^\top P)\bar{\vec z} = 0$. In this case, we have
\begin{align*}
\dot V(\bar{\vec z}) = \bar{\vec z}^\top(P\Lambda+\Lambda^\top P)\bar{\vec z} \geq 0
\end{align*}
for any input $u$, which contradicts the definition of a CLF. This completes the proof.
\end{proof}

Following convex optimization formulation can be formulated to search for quadratic Lyapunov function for bilinear system without uncertainty in Eq. (\ref{bilinear_sys}).
\begin{eqnarray}
\label{opt}
\minimize_{t > 0, \ P = P^\top} & \quad t - \gamma{\rm Trace}(P B) \nonumber \\
\subjt & \quad tI - (P\Lambda+\Lambda^\top P) \succeq 0 \nonumber \\
& \quad c^\text{max} I \succeq P \succeq c^\text{min} I
\end{eqnarray}
where $c^\text{max} > c^\text{min} > 0$, respectively, are two given positive scalars forming bounds for the largest and the least eigenvalues of $P$. The variable $t$ here represents an epigraph form for the largest eigenvalue of $P\Lambda+\Lambda^\top P$.

Optimization (\ref{opt}) has combined two objectives. On the one hand, we minimize the largest eigenvalue of $P\Lambda+\Lambda^\top P$. On the other hand, we try to maximize the least singular value of $PB+B^\top P$ the same time. Noticing that it may be difficult to maximize the least singular value of $PB+B^\top P$ directly, we maximize the trace of $PB$ instead and employ a parameter $\gamma > 0$ to balance these two objectives.

\begin{remark}
When an optimal $P^\star$ is solved from (\ref{opt}), we still need to check whether it satisfies the condition of Theorem \ref{bilin_stb} or not. So if one $P^\star$ fails the condition check, then we may tune the parameter $\gamma$ and solve the above optimization again until we obtain a correct $P^\star$. Nevertheless, we observe from simulations (see the multiple examples in our simulation section) that when we choose a $\gamma = 2$, optimization (\ref{opt}) will always yield an optimal $P^\star$ that satisfies the condition of Theorem \ref{bilin_stb}. 
\end{remark}

\begin{remark}
We also need to point out that, compared to searching for a nonlinear CLF for the original nonlinear system (\ref{non_lin_sys}), the procedure for seeking a quadratic CLF for the bilinear system (\ref{bilinear_sys}) becomes quite easier and more systematic. Furthermore, a quadratic CLF for the bilinear system is, in fact, non-quadratic (i.e., contains higher order nonlinear terms) for the system (\ref{non_lin_sys}).
\end{remark}

Once a quadratic control Lyapunov function $V(\vec z) = \vec z^\top P\vec z$ is found for bilinear system (\ref{bilinear_sys}), we have several choices for designing a stabilizing feedback control law. For instance, applying the control law (\ref{control2}) or (\ref{control1}) we can construct
\begin{align}
k(\vec z) &= -\beta_k \sgn\big[\vec z^\top(PB+B^\top P)\vec z \big] \\
k(\vec z) &= -\beta_k \vec z^\top(PB+B^\top P)\vec z.\label{control_formula}
\end{align}
Moreover, given a positive semidefinite cost $q(\vec z) \geq 0$, we may also apply the inverse optimality property to design an optimal control via Sontag's formula (\ref{mod_sontag}) to obtain
\begin{align}
k(\vec z) = \begin{cases} -\frac{\vec z^\top(P\Lambda+\Lambda^\top P)\vec z + \sqrt{(\vec z^\top(P\Lambda+\Lambda^\top P)\vec z)^2 + q(x)(\vec z^\top(PB+B^\top P)\vec z)^2}}{\vec z^\top(PB+B^\top P)\vec z} & \text{if } \vec z^\top(PB+B^\top P)\vec z \neq 0 \\ 0 & \text{otherwise}. \end{cases} \label{sontag_formula}
\end{align}
Following algorithm can be outlined for the design of stabilizing feedback controller from time-series data.
\begin{algorithm2e}[h]\label{algorithm}
\SetKwBlock{Phaseone}{Phase I: Modeling}{end}
\SetKwBlock{Phasetwo}{Phase II: Optimization}{end}
\SetAlgoLined
\KwData{Given open-loop time-series data $\{\vec{x}_k^0\}=\{\vec{x}_0^0,\vec{x}_1^0,\ldots,\vec{x}_M^0\}$, and $\{\vec x_k^1\}$ with $s=1$ in (\ref{eq_s}) both with Gaussian process noise added}
\KwResult{Feedback control $u=k(\vec z)$}
\Phaseone{
Choose $N$ dictionary functions $\varPsi(\vec x):=\begin{bmatrix}\psi_1(\vec{x}) & \psi_2(\vec{x}) & \cdots & \psi_N(\vec{x})\end{bmatrix}^\top$.

\For{$\vec{x}_i, \;i = 0,1,2,\ldots,M$}{
$\varPsi(\vec x_i):=\begin{bmatrix}\psi_1(\vec{x}_i) & \psi_2(\vec{x}_i) & \cdots & \psi_N(\vec{x}_i)\end{bmatrix}^\top$
}
Obtain $\vec G^0$ and $\vec A^0$ matrices
${\vec G^0}=\frac{1}{M} \sum_{m=1}^M \varPsi(\vec{x}_m) \varPsi(\vec{x}_m)^\top$;
${\vec A^0}=\frac{1}{M} \sum_{m=0}^{M-1} \varPsi(\vec{x}_m)\varPsi(\vec{x}_{m+1})^\top$.

Compute $\vec U^0 = ({\vec G}^0)^\dagger {\vec A}^0$, and its eigenfunctions $\phi_j=\varPsi^\top v_j$, where $v_j$ is the $j$th eigenvector of $\vec U^0$ with respect to eigenvalue $\lambda_j$, $j=1,2,\ldots,N$.

Convert to continuous time eigenvalues $\hat\lambda_i = \log(\lambda_i)/\Delta t$

Get $\varLambda = diag(\hat\lambda_1,\hat\lambda_2,\ldots,\hat\lambda_N)$ by block diagonalization of eigenvalues $\lambda_i$, use (\ref{eig_conversion}) if $i$, ${i+1}$ complex conjugate.

Obtain the new eigenfuntion $\hat\varPhi(\vec x)$ similarly, where $\hat \phi_i:=\phi_i$ if $\phi_i$ is a real-valued and $\hat \phi_i:=2 {\rm Re}(\phi)$, $\hat \phi_{i+1}:=-2{\rm Im}(\phi_i)$, if $i$ and $i+1$ are complex conjugate.

Replace the dictionary function $\varPsi(\vec x)$ with $\vec z = \hat{\varPhi}(\vec x)$ and repeat Step $2$ to $7$ with the datasets $\{\vec{x}_k^0\}$ and $\{\vec{x}_k^1\}$ to get $\bar{\vec{U}}^0$ and $\bar{\vec{U}}^1$.

Get $B = (\bar{\vec{U}}^1-\bar{\vec{U}}^0)/\Delta t$ }

\Phasetwo{
Solve the following convex problem for optimal $P*$ with $\Lambda$ and $B$,
\begin{eqnarray*}
\minimize_{t > 0, \ P = P^\top} & \quad t - \gamma{\rm Trace}(P B) \nonumber \\
\subjt & \quad tI - (P\Lambda+\Lambda^\top P) \succeq 0 \nonumber \\
& \quad c^\text{max} I \succeq P \succeq c^\text{min} I
\end{eqnarray*}
where $c^\text{max} > c^\text{min} > 0$, $\gamma>0$ are chosen properly.
}
Feedback control $u=k(\vec z)= -\beta_k \vec z^\top(PB+B^\top P)\vec z$ or modified Sontag's formula,
\begin{align*}
k(\vec z) = \begin{cases} -\frac{\vec z^\top(P\Lambda+\Lambda^\top P)\vec z + \sqrt{(\vec z^\top(P\Lambda+\Lambda^\top P)\vec z)^2 + q(x)(\vec z^\top(PB+B^\top P)\vec z)^2}}{\vec z^\top(PB+B^\top P)\vec z} & \text{if } \vec z^\top(PB+B^\top P)\vec z \neq 0 \\ 0 & \text{otherwise}. \end{cases}
\end{align*}
\caption{Data-driven Stabilization Controller design framework}
\end{algorithm2e}


\section{Simulation Results}\label{section_simulation}

\noindent{\bf Example 1: Duffing Oscillator}\\
\noindent The first example we present is for the stabilization of duffing oscillator. The controlled duffing oscillator equation is written as follows.
\begin{eqnarray}\label{duffing_sys}
\dot{x}_1 &=& x_2\\\nonumber
\dot{x}_2 &=& (x_1-x_1^3)-0.5x_2+u.
\end{eqnarray}
The uncontrolled equation for duffing oscillator consists of three equilibrium points, two of the equilibrium points at $(\pm 1,0)$ are stable, and one equilibrium point at the origin is unstable. For identification of the control system dynamics, we excite the system with white noise with zero mean and $0.01$ variance. The continuous time control equation is discretized with a sampling time of $\Delta t=0.25 s$. In Fig. \ref{fig:duffing}a, we show the sampling complexity plot for the approximation error as the function of data length. As proved in \cite{sample_complexity}, the error for the approximation of the $\Lambda$ and $B$ matrix decreases as $\frac{1}{\sqrt{T}}$, where $T$ is a data length. The error plot in Fig. \ref{fig:duffing}a satisfies this rate of decay. The sample complexity results in Fig. \ref{fig:duffing}a are obtained using ten randomly chosen initial condition and generating time-series data over the different length of time ranging from six-time steps to 30-time steps. For each fixed time step we compute the $\Lambda$ and $B$ matrices. The error $\parallel\Lambda-\bar \Lambda\parallel_2$ and $\parallel B-\bar B\parallel_2$ is computed  at each fixed time step where $\bar \Lambda$ and $\bar B$ are computed using data collected over $50$ time steps. The dictionary function used in the approximation of the Koopman operator has a maximum degree of five, i.e., $21$ basis function, $N=21$. In particular, following choice of dictionary function is made in the approximation.

\[{\varPsi}({\bf x}) = [1,\;x_1,\;x_2,\;x_1x_2,\;\ldots,\;x_1^5,\;x_1^4x_2,\;x_1^3x_2^2,\;x_1^2x_2^3,\;x_1x_2^4,\;x_2^5]\]

For control design, we use an approximation of $\Lambda$ and $B$ matrices computed over $30$ time steps. The controller is designed using the Algorithm \ref{algorithm}. 
For this duffing oscillator example, we use a control design formula in Eq. (\ref{control_formula}). To verify the effectiveness of the designed controller we simulate the closed loop system with the  \texttt{ode15s} solver in \textbf{MATLAB} starting from $10$ randomly chosen initial conditions within the region $[-1.5,1.5]\times [-1,1]$. In Fig. \ref{fig:duffing}c, we show the closed loop trajectories in red starting from different initial conditions overlaid on the open loop trajectories in blue. We notice that the controller force the trajectories of the closed-loop system along the stable manifold of the open loop system before the trajectories slide to the origin. The control plots from different initial conditions are shown in Fig. \ref{fig:duffing}c.

\begin{figure}[!htp]
\begin{framed}
\centering
\subfigure[]{\includegraphics[width=1.9in]{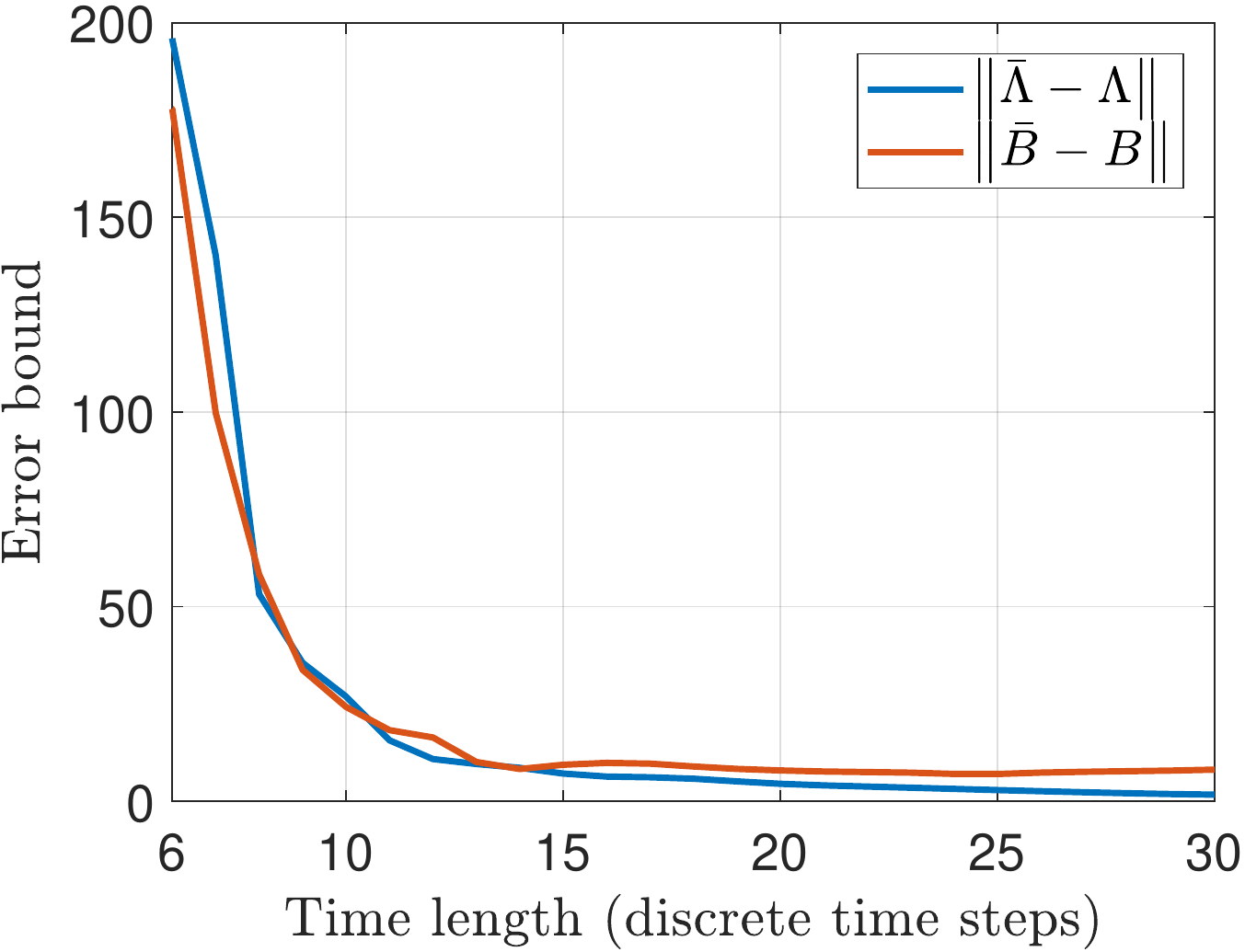}}\qquad
\subfigure[]{\includegraphics[width=1.9in]{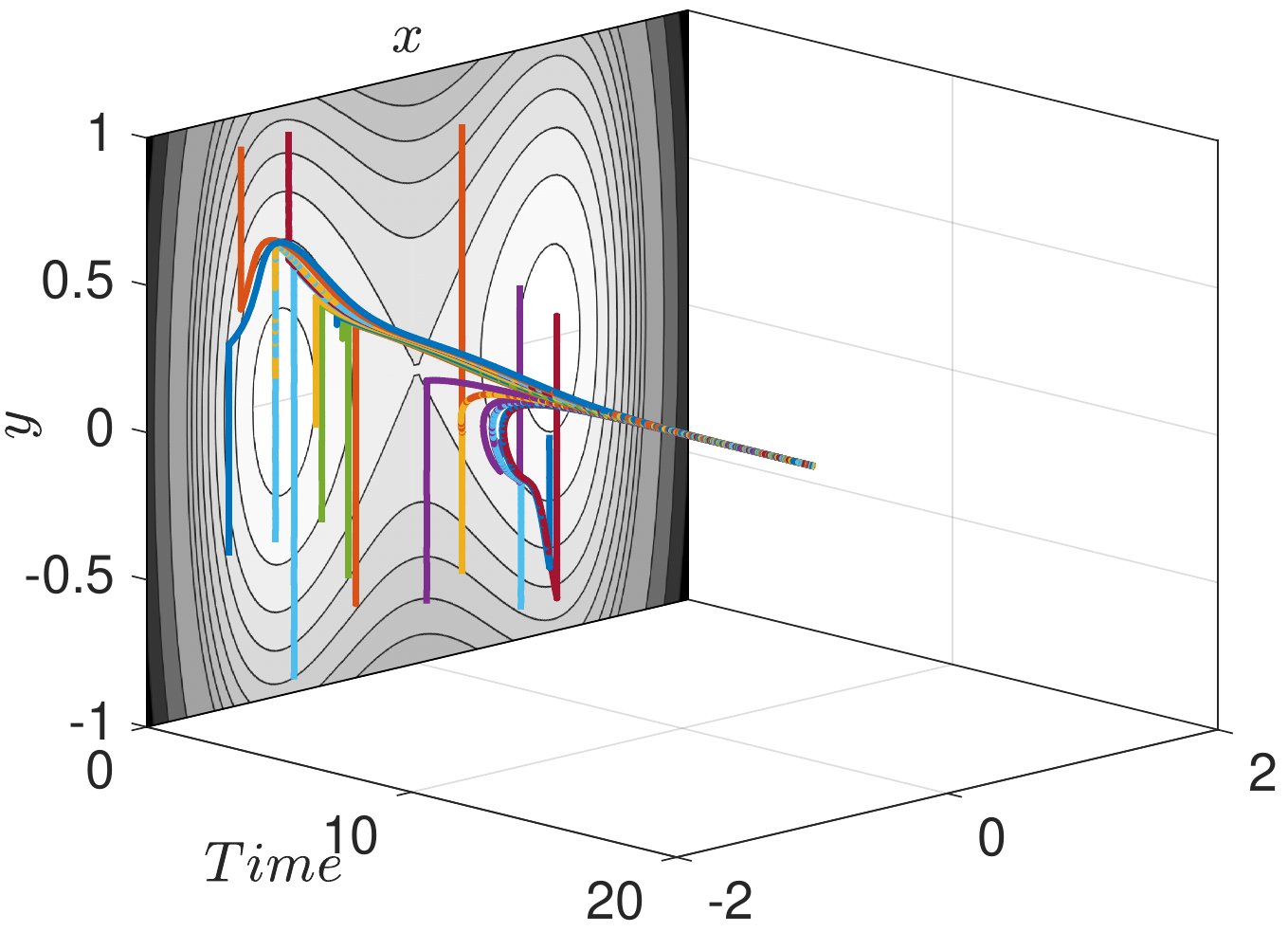}}

\subfigure[]{\includegraphics[width=1.9in]{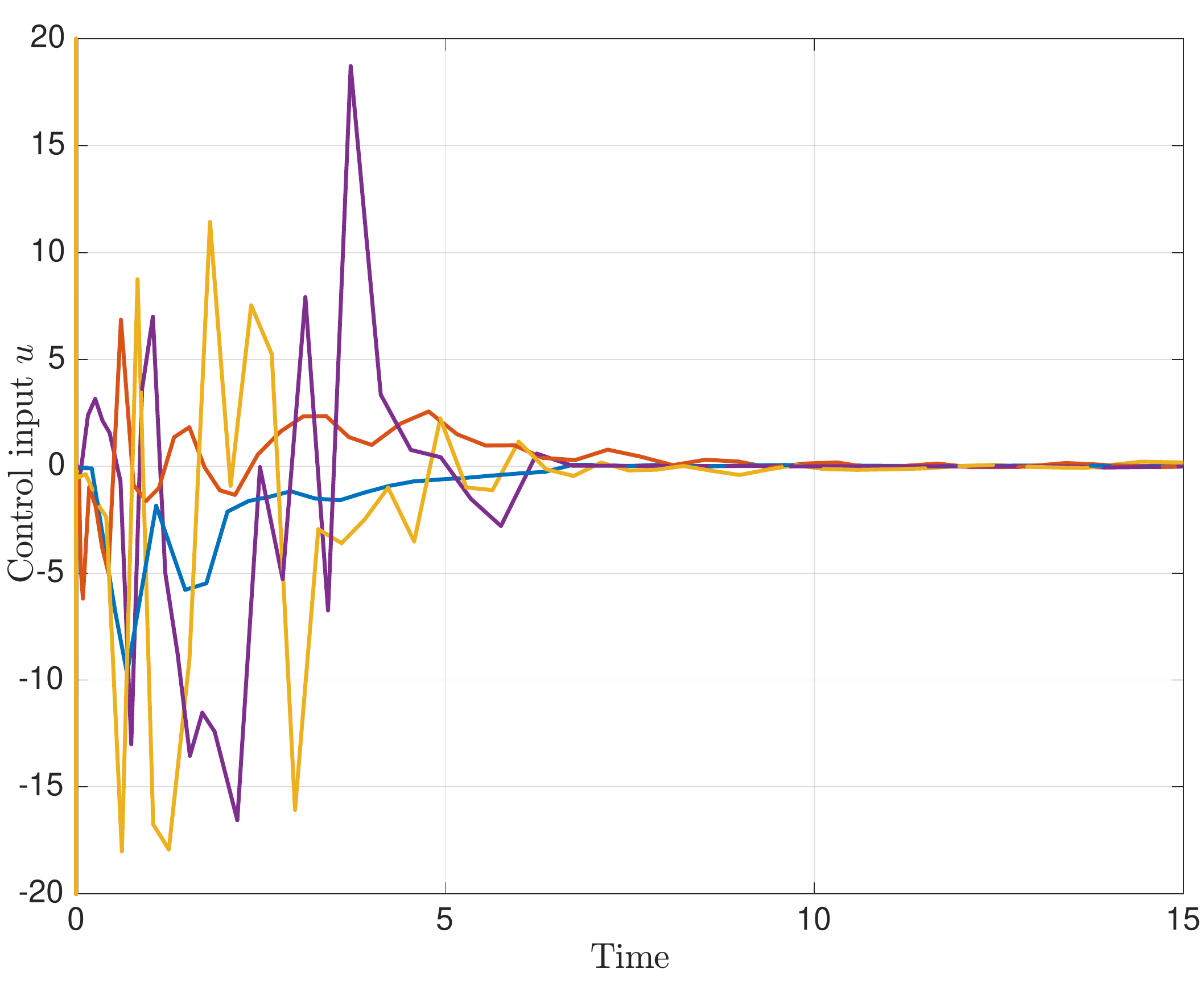}}\qquad
\subfigure[]{\includegraphics[width=1.9in]{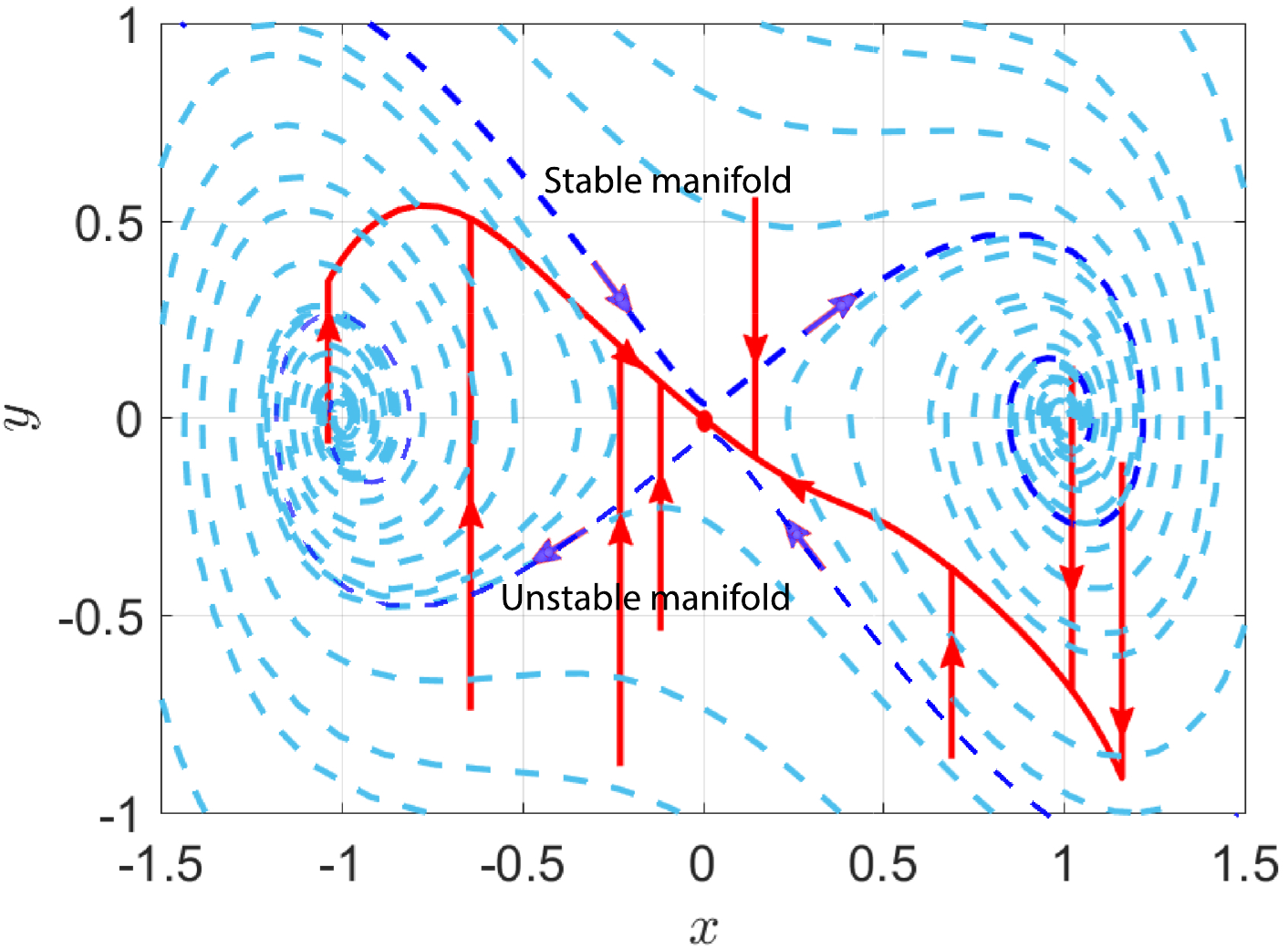}}
\caption{Data-driven stabilization of Duffing oscillator. a) Sample complexity error bounds for the approximation of $\Lambda$ and $B$ matrices as the function of data length; b) Closed-loop trajectories vs time from multiple initial conditions; c) Control value vs time from different initial conditions; d) Comparison of closed loop and open loop trajectories in state space.}
\label{fig:duffing}
\end{framed}
\end{figure}

\[\]

\noindent{\bf Example 2: Lorenz System}\\
\noindent The second example we pick is that of Lorentz system. The control Lorentz system can be written as follows
\begin{eqnarray}\label{L}
\dot{x}_1 &=& \sigma(x_2-x_1)\\\nonumber
\dot{x}_2 &=& x_1(\rho-x_3)-x_2+u\\\nonumber
\dot{x}_3 &=& x_1x_2 - \beta x_3
\end{eqnarray}
where $\vec x\in\mathbb{R}^3$ and $u\in\mathbb{R}$ is the single input. With the parameter values of $\rho=28$, $\sigma=10$, $\beta=\frac{8}{3}$, and control input $u=0$ the Lorenz system exhibit chaotic behavior. In this 3D example, we generated the time-series data from $1000$ random chosen initial conditions and propagate each of them for $T_{final}=10s$ with sampling time $\Delta t = 0.001s$. For the purpose of identification the system is excited with white noise input with zero mean and $0.01$ variance. The dictionary functions ${\varPsi}(\boldsymbol{\bf x})$ consists of 20 monomials of most degree $D = 3$.
\[\boldsymbol{\varPsi}(\boldsymbol{\bf x}) = [1,\;x_1,\;x_2,\;x_3,\;\ldots,\;x_1^3,\;x_1^2x_2,\;x_1^2x_3,\;x_1x_2x_3,\;\ldots\;x_3^3]\]
The objective is to stabilize one of the critical points $(\sqrt{\beta(\rho-1)},\sqrt{\beta(\rho-1)},\rho-1)$ of the Lorentz system. The system is stabilized using the control formula in Eq. (\ref{control_formula}).
To validate the closed loop control designed using the Algorithm \ref{algorithm}, we perform the closed loop simulation with five randomly chosen initial conditions in the domain  $[-5,5]\times[-5,5]\times[0,10]$ and solve the closed-loop system with \texttt{ode15s} solver in \textbf{MATLAB}. In Fig.~\ref{fig:lorenz}a , we show the open loop and closed loop trajectories starting from five different initial conditions and converging to the critical point.

\[\]

\begin{figure}[!htp]
\begin{framed}
\centering
\subfigure[]{\label{noerror}\includegraphics[width=1.9in]{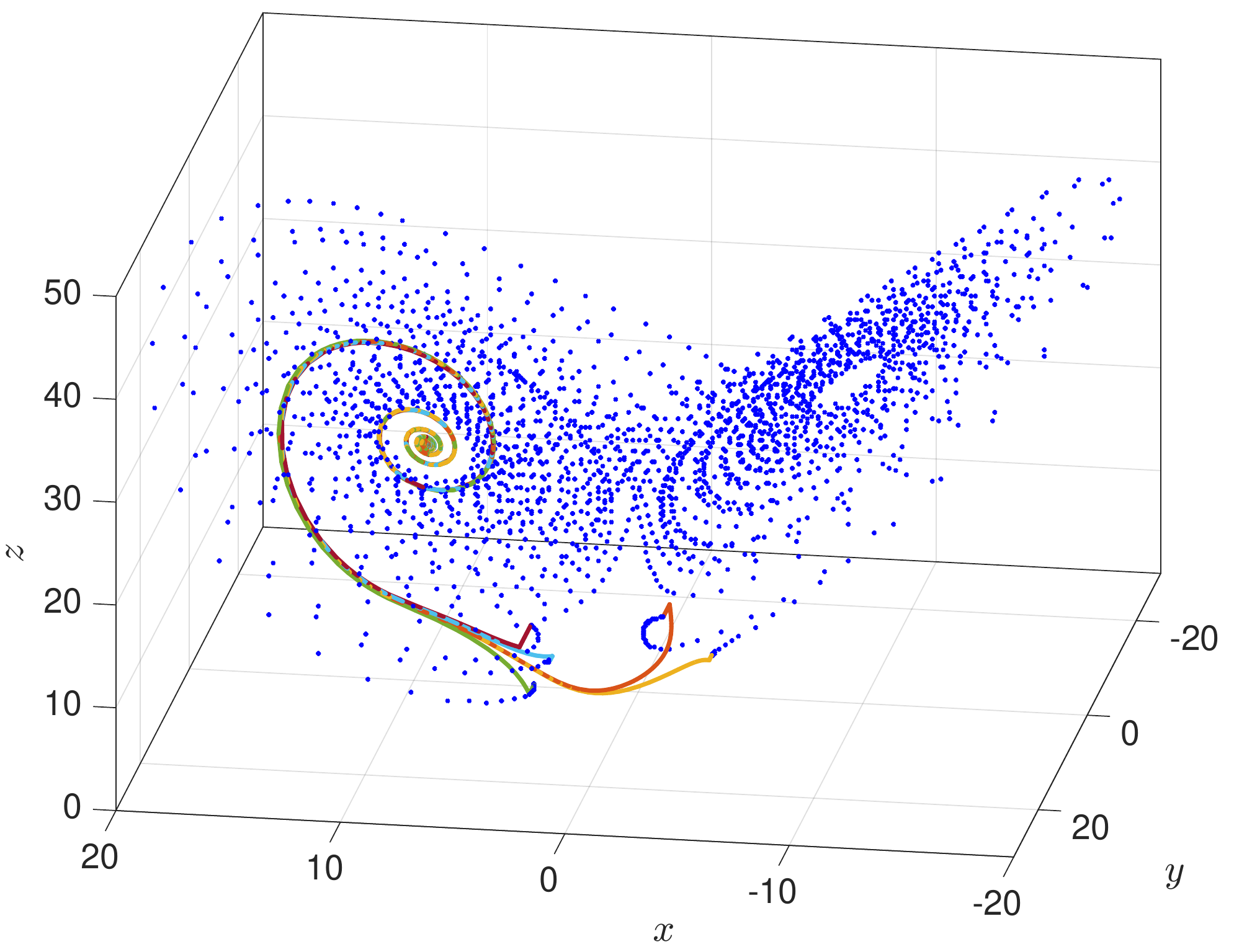}}\qquad
\subfigure[]{\label{synerror}\includegraphics[width=1.9in]{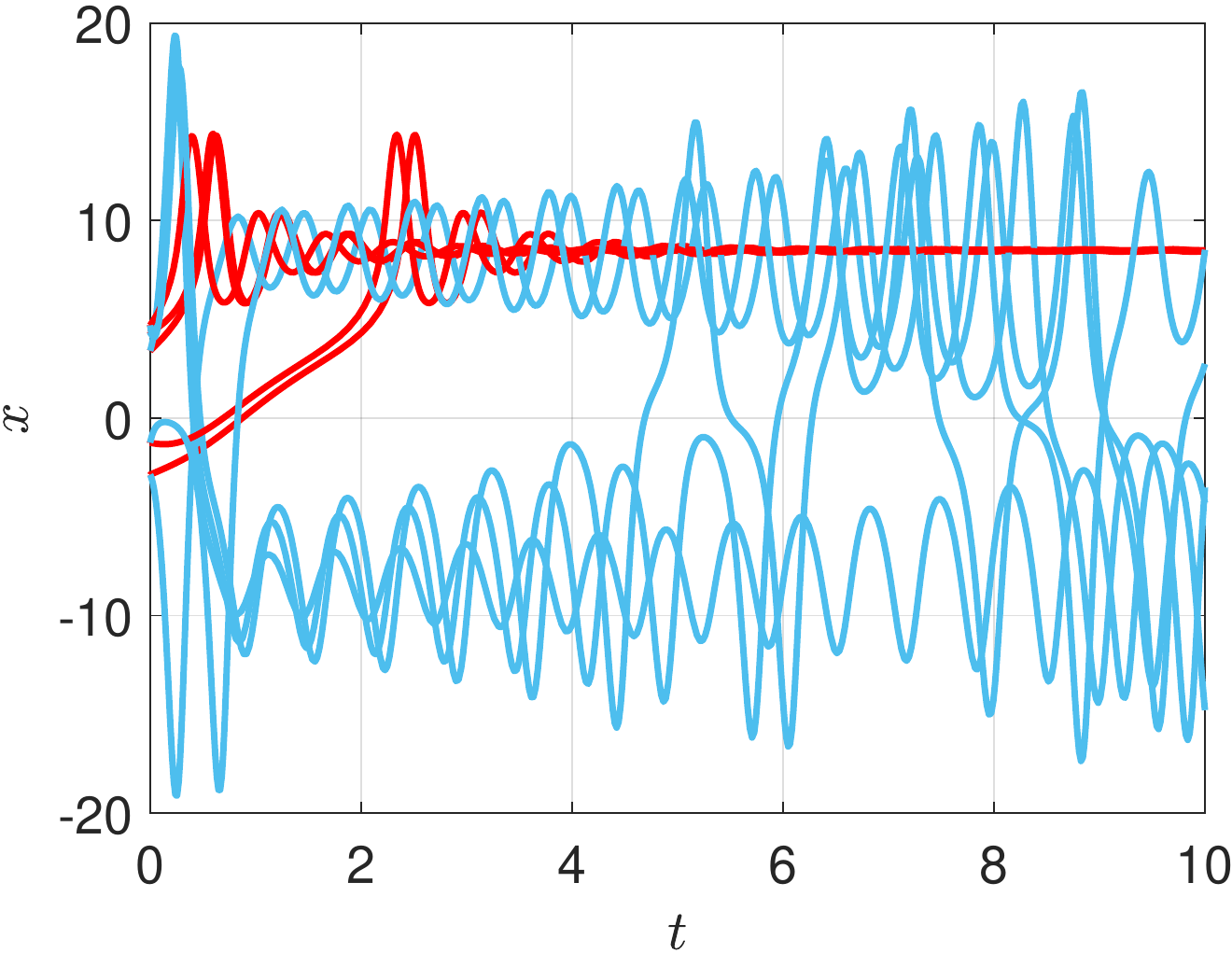}}

\subfigure[]{\label{deadlock}\includegraphics[width=1.9in]{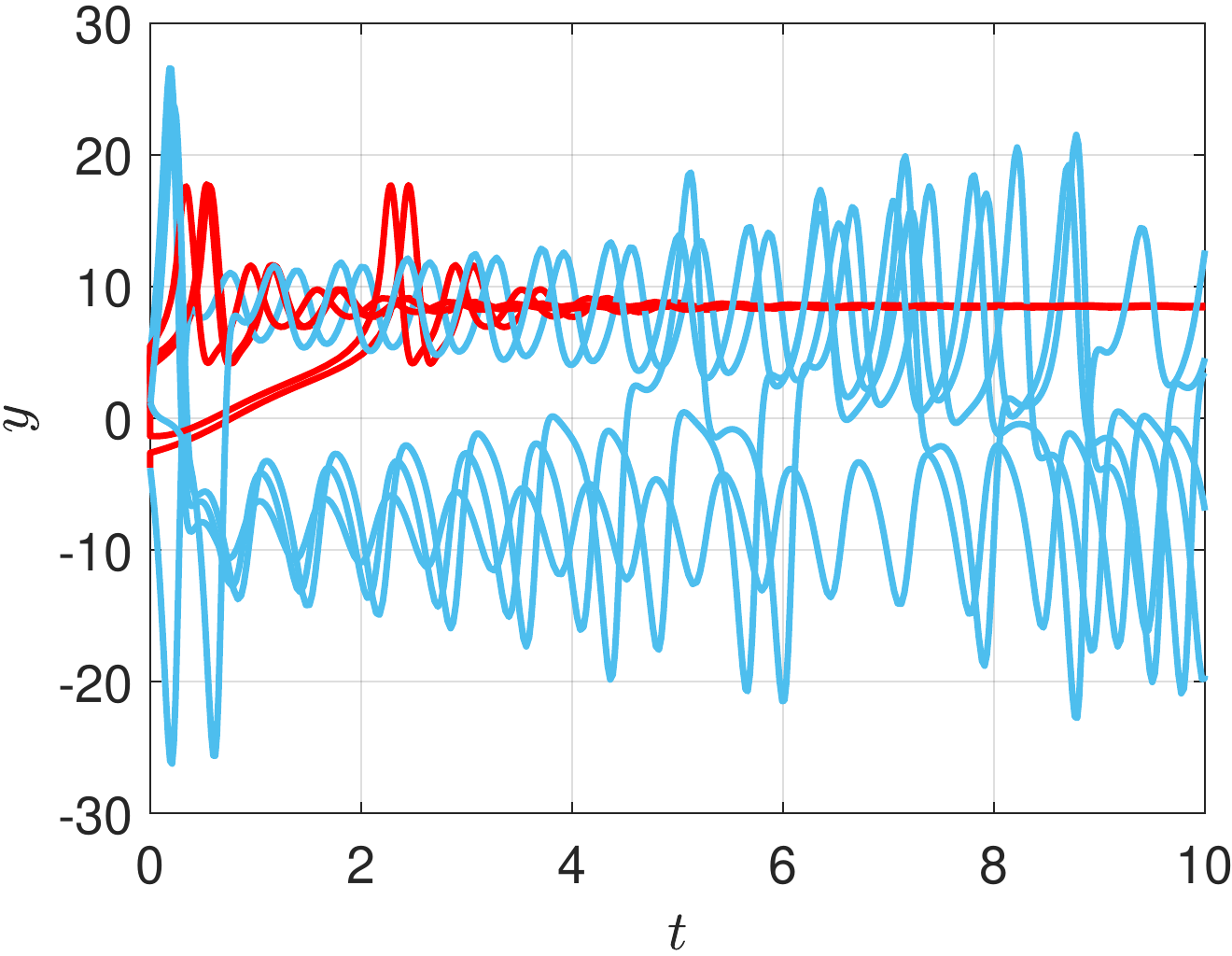}}\qquad
\subfigure[]{\label{lacksync}\includegraphics[width=1.9in]{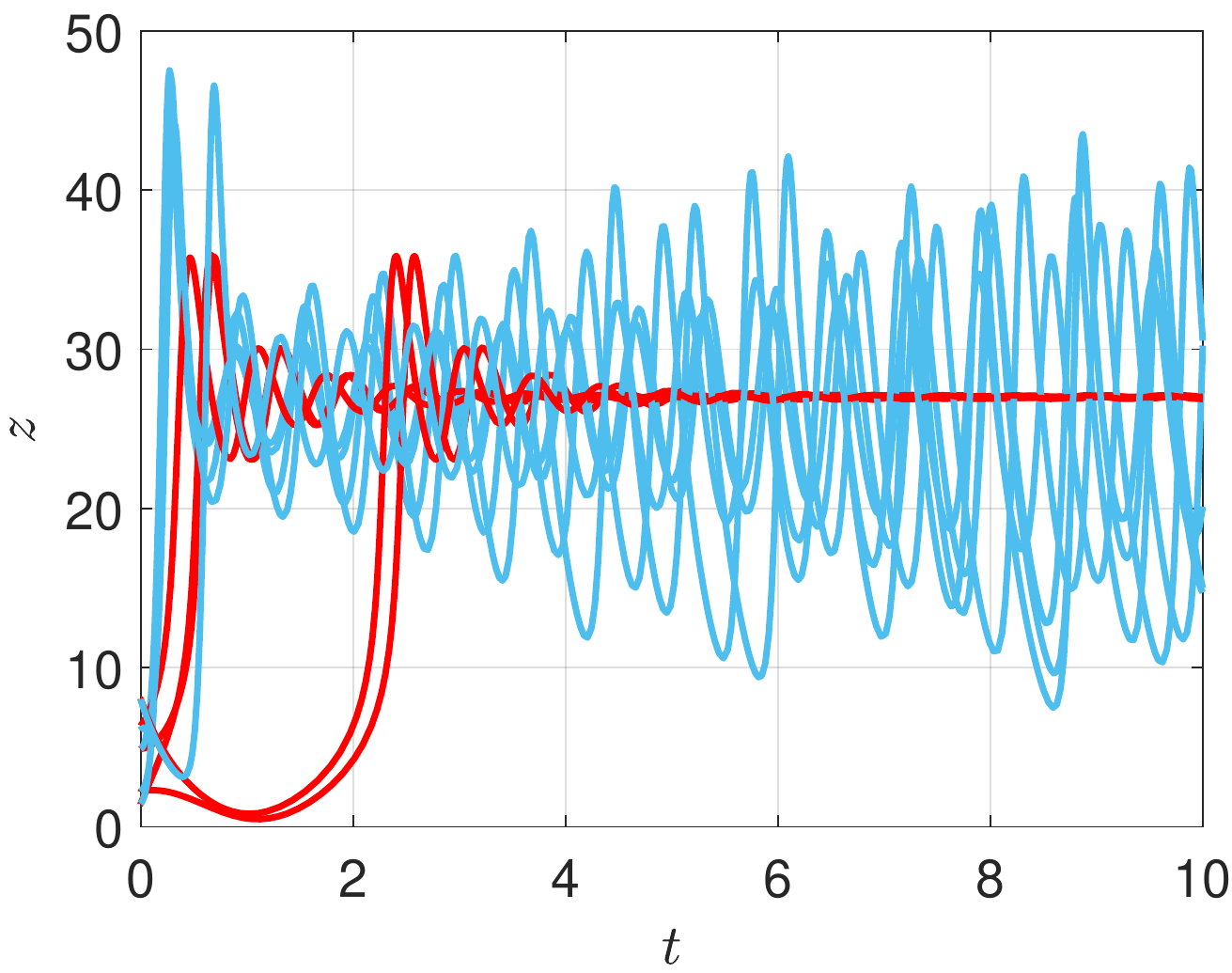}}
\caption{Feedback Stabilization of Lorentz system. a) Comparison of open loop and closed loop trajectories in state space; b) $x(t)\; vs $ time, open loop (blue) and closed loop (red); c) $y(t)\; vs $ time, open loop (blue) and closed loop (red); d) $z(t)\; vs $ time, open loop (blue) and closed loop (red).}
\label{fig:lorenz}
\end{framed}
\end{figure}

\noindent{\bf Example 3: IEEE 9 bus Power System}

\noindent In the last example, we consider the IEEE 9 bus system, the line diagram of which is shown in Fig.~\ref{fig:9bus}a. The model we are using is based on the modified 9 bus test system in ~\cite{Sauer_pai_book}. The system consists of 3 synchronous machines(generators) with IEEE type-I exciters, loads and transmission lines.
The synthetic data is generated using PST (Power System Toolbox) in MATLAB~\cite{207380}, the 9 bus power system network can be described by a set of differential algebraic equations (DAE), consider a power system model with $n_g$ generator buses and $n_l$ load buses. the closed-loop generator dynamics for the $i$th generator bus can be represented as a $2^{nd}$ order dynamical model with the control $u$:
\begin{eqnarray}\label{generator_dynamic_eq}
\begin{aligned}
&\frac{d\delta_i}{dt}  = \omega_i - \omega_s \\
&\frac{d\omega_i}{dt}  = \frac{1}{M_i}\left(P_{m_i}-\sum_{j\in {\cal N}_i}\frac{E_i E_j}{X_{ij}}\sin(\delta_i-\delta_j)- D_i (\omega_i-\omega_s)\right)+u_i \\
\end{aligned}
\end{eqnarray}
where $\delta_i$, $\omega_i$ are the dynamic states of the generator and correspond to the generator rotor angle, the angular velocity of the rotor. The values for the other parameters is chosen as follows: $\omega_s=1$, the generator mass $M_i = 23.64,\; 6.4,\;3.1$, the internal damping $D_i = 0.05,\;0.95,\;0.05$, the generator power $P_{m_i}= 0.719,\;1.63,\;0.85$ for $i=1,2,3$. The values of $X_{ij}$ are taken from the PST in MATLAB.


%


For the approximation of Koopman operator and eigenfunctions, the time-series data are generated from 100 initial conditions. Each initial condition are propagated for $T_{final}=10s$ and $\Delta t=0.01s$. The dictionary function $H(x)$ in this example are chosen as 84 monomials of most degree $D=3$. The data-driven stabilizing control is designed using modified Sontag's formula control in Eq. (\ref{mod_sontag}), where $q(x) = 10x^\top x$. The simulation results for this example are shown in Fig. \ref{fig:9bus}. We notice that the open loop system is marginally stable with sustained oscillations. The objective of the stabilizing controller is to stabilize to frequencies to $\omega_s=1$ and the point for the stabilization of $\delta$ dynamics is determined by $P_{m_i}$. Simulation results show that the data-driven stabilizing controller is successful in stabilizing the power system dynamics.


\begin{figure}[!htp]
\begin{framed}
\centering
\subfigure[]{\label{noerror}\includegraphics[width=2.2in]{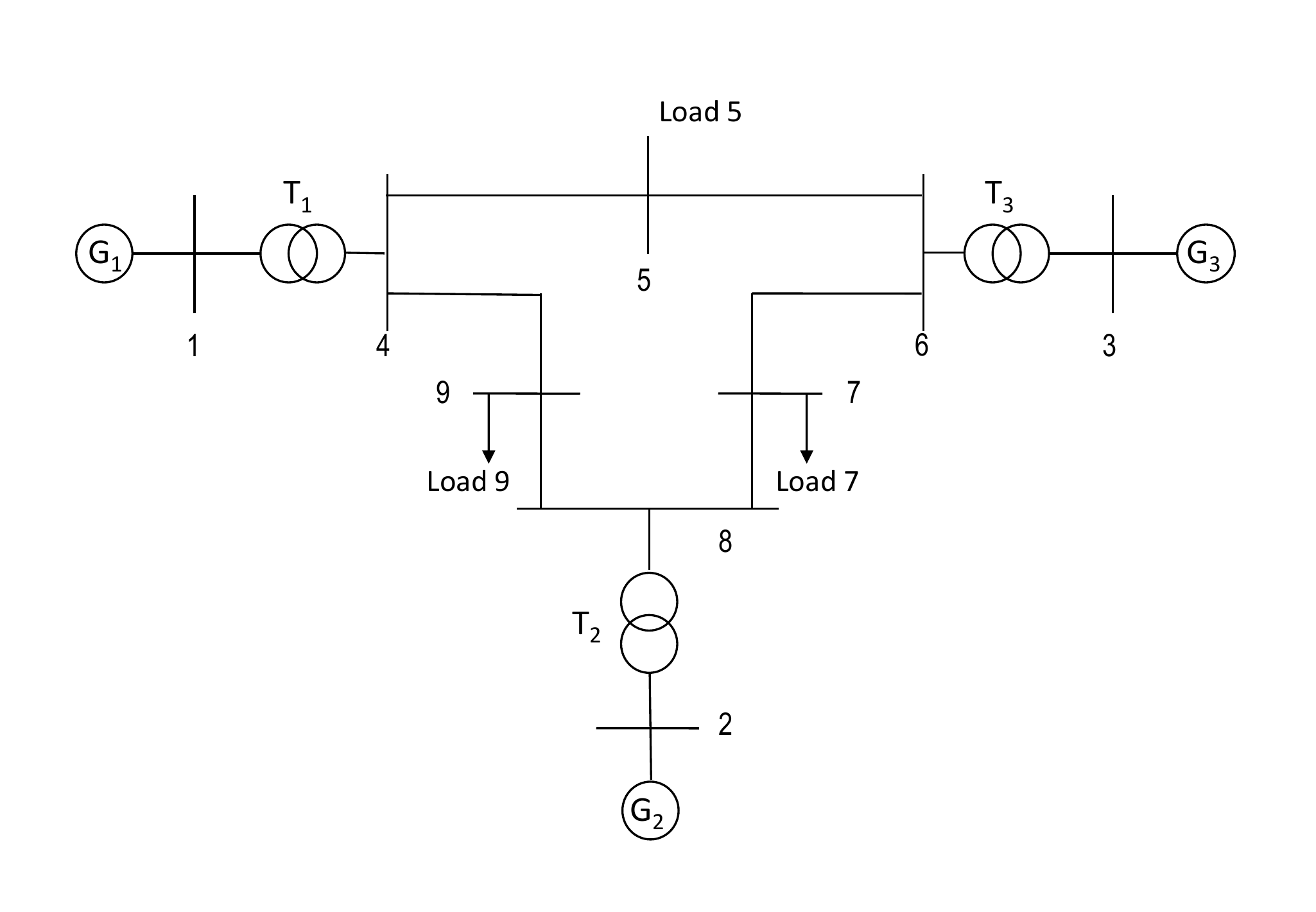}}\qquad
\subfigure[]{\label{synerror}\includegraphics[width=1.9in]{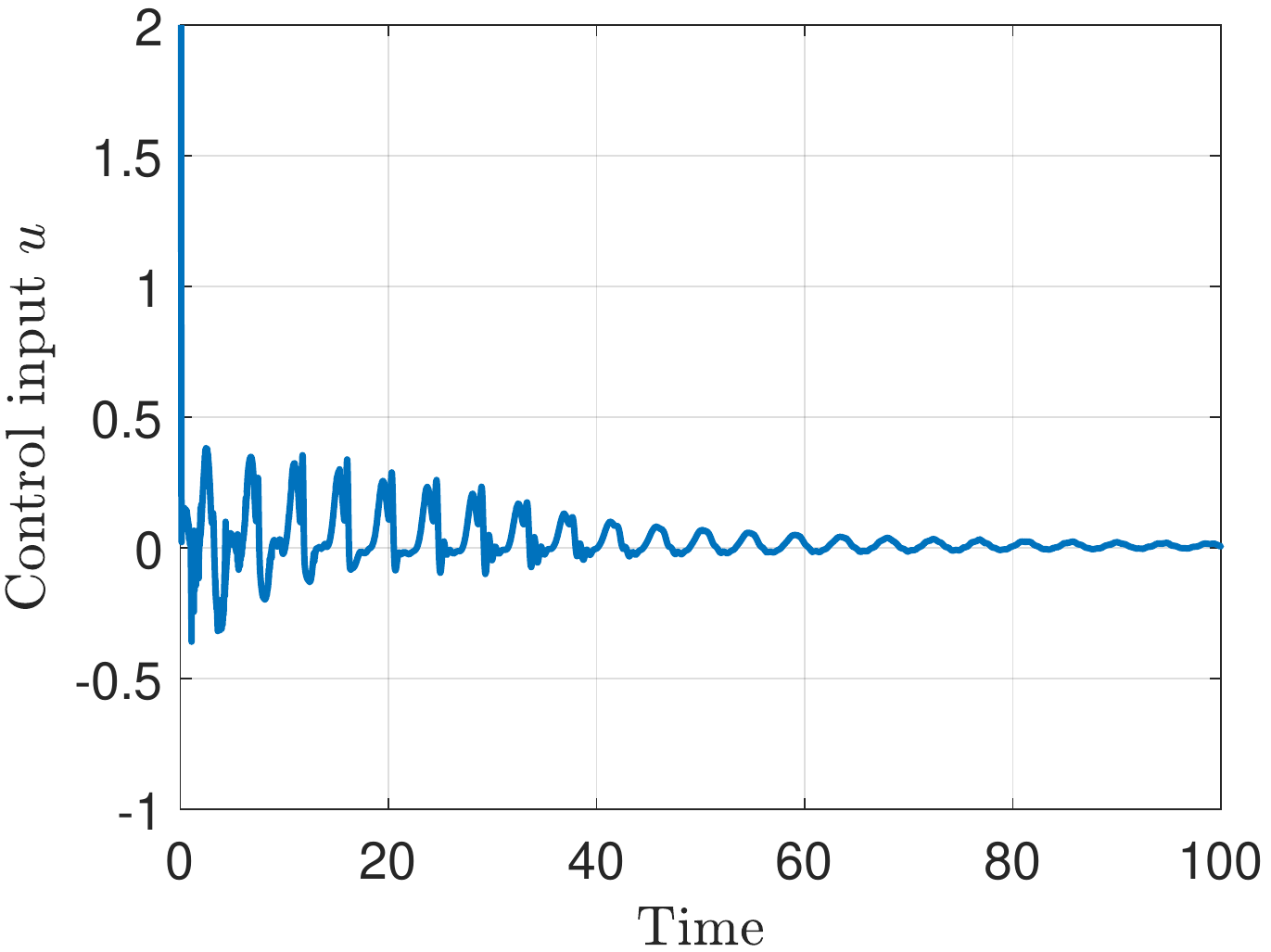}}

\subfigure[]{\label{deadlock}\includegraphics[width=1.9in]{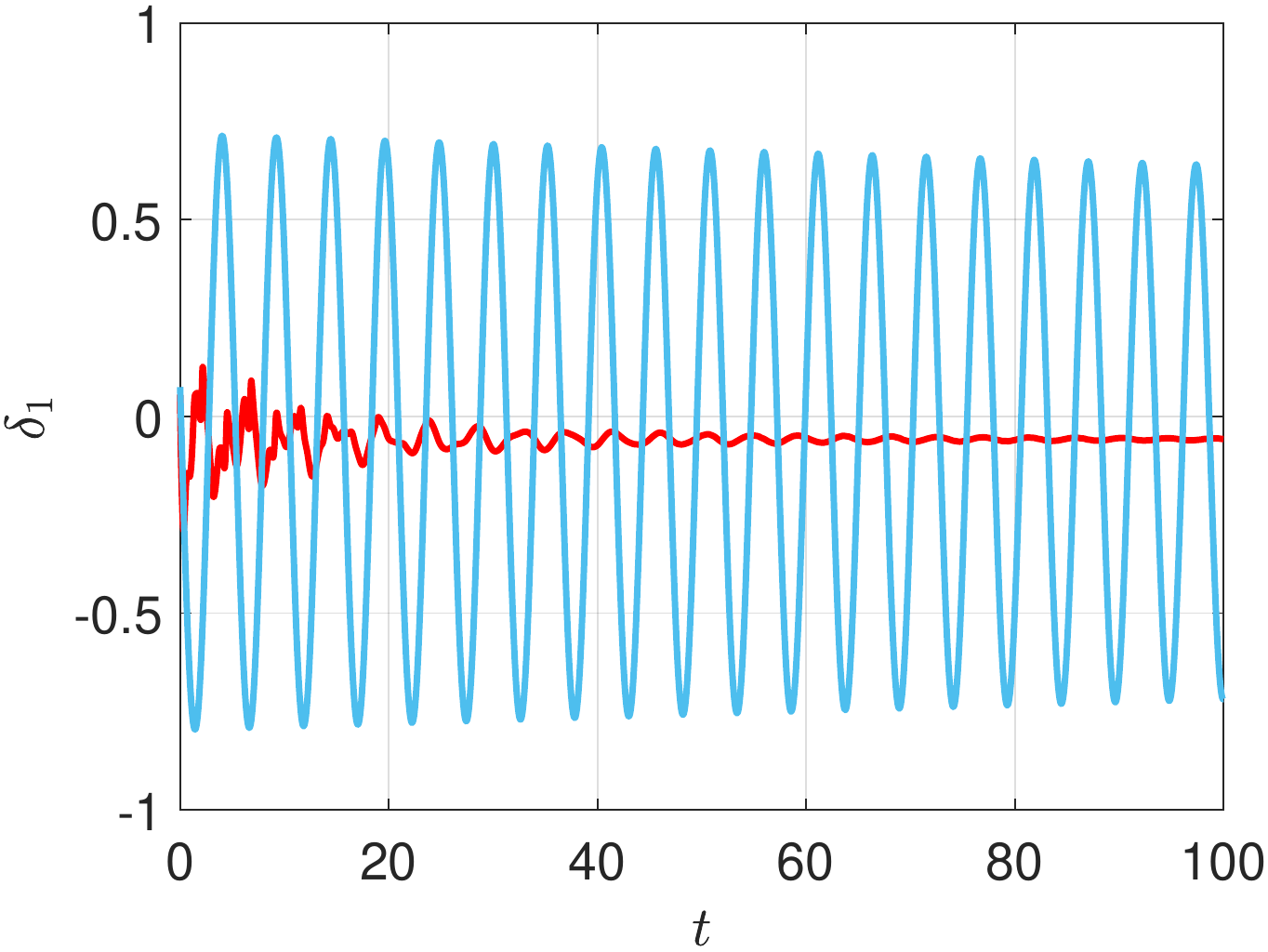}}\qquad
\subfigure[]{\label{lacksync}\includegraphics[width=1.9in]{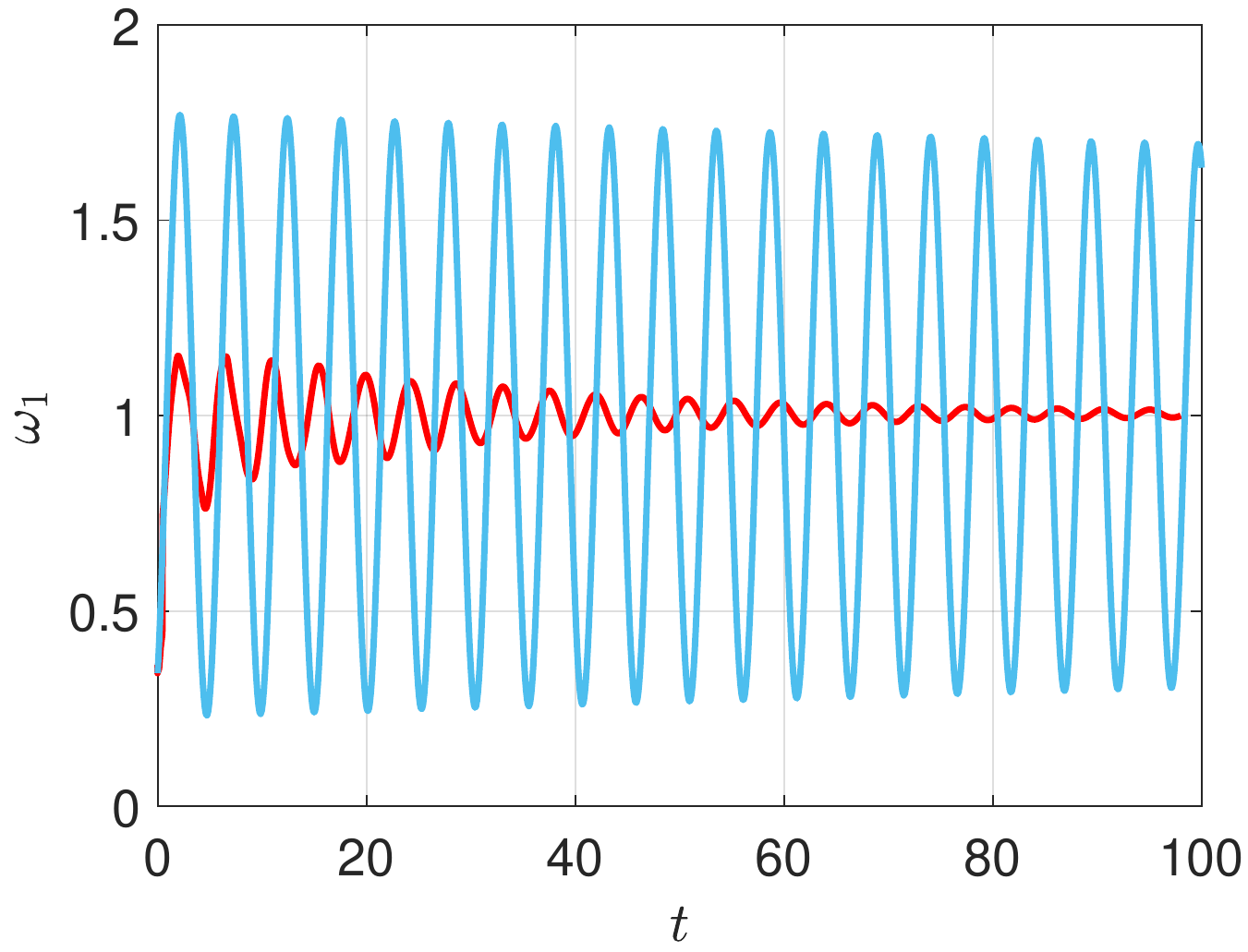}}
\caption{Stabilization of IEEE nine bus system. a) Line diagram for IEEE nine bus system; b) Control value vs time; c) Comparison of open loop and closed loop trajectory for phase angle $\delta_1(t)$ of generator 1; d) Comparison of open loop and closed loop trajectory for frequency $\omega_1(t)$ of generator 1.
}\label{fig:9bus}
\end{framed}
\end{figure}


\section{Conclusion}\label{section_conclusion}
In this chapter, we provided a systematic approach for the data-driven feedback stabilization of nonlinear control systems. A data-driven approach is proposed for the identification of nonlinear control system and control Lyapunov function-based stabilizing feedback controller. The bilinear structure of the control system in Koopman eigenfunction coordinate is exploited to provide a convex optimization-based approach for the search of control Lyapunov function. Simulation results are presented to verify the applicability of the developed framework. 
\bibliographystyle{spmpsci.bst}
\bibliography{ref,ref1}
\end{document}